\documentclass[12pt]{article}
\usepackage{amsmath, amssymb, amsthm}
\usepackage{geometry}
\usepackage{enumitem}
\usepackage{graphicx,graphics}
\usepackage{hyperref}
\usepackage{amsfonts}

\newcommand{\len}{\operatorname*{length}}
\newcommand{\bd}{\operatorname*{bd}}
\newcommand{\aff}{\operatorname*{aff}}

\newcommand{\R}{\mathbb{R}}
\newcommand{\Rt}{\mathbb{R}^3}
\newcommand{\Rn}{\mathbb{R}^{n}}

\newcommand{\Suno}{\mathbb{S}^{1}}

\title{Convex bodies with pairs of sections associated by reflections}
\author{E. Morales-Amaya}
\date{}

\newtheorem{theorem}{Theorem}
\newtheorem{lemma}{Lemma}
\newtheorem{corollary}{Corollary}

\newtheorem{remark}{Remark} 

\begin{document}
	
	\maketitle
	
	\begin{abstract}
		In this work we prove that if for a pair of convex bodies $K_1, K_2 \subset \mathbb{R}^n$, $n \geq 3$, there exists a hyperplane $H$ and two distinct points $p_1$ and $p_2$ in $\mathbb{R}^n \setminus H$ such that for every $(n-2)$-plane $M \subset H$, there exists a reflection mapping the hypersection of $K_1$ defined by $\mathrm{aff}\{p_1, M\}$ onto the hypersection of $K_2$ defined by $\mathrm{aff}\{p_2, M\}$, then there exists a reflection which maps $K_1$ onto $K_2$.
	\end{abstract}
	
\section{Introduction}

We begin by formulating two motivating questions.

Let $K_1$ and $K_2$ be convex bodies in $\mathbb{R}^n$ with $n \geq 3$, and let $p_1 \neq p_2$ be two distinct points. Suppose that for every hyperplane $\Pi$ passing through the origin in $\mathbb{R}^n$, there exists an affine transformation $T_\Pi : \mathbb{R}^n \to \mathbb{R}^n$ such that
\begin{equation}
	T_\Pi\big((\Pi + p_2) \cap K_1\big) = (\Pi + p_1) \cap K_2,
\end{equation}
and that the collection $\{T_\Pi\}$ forms a subgroup $\Omega$ of the affine group $\mathcal{A}(n)$ (which can be defined as the semi direct product of the orthogonal group $O(n)$ and $\mathbb{R}^n$).

\textbf{Question 1.} For which convex bodies $K_1, K_2 \subset \mathbb{R}^n$ and subgroups $\Omega \subset \mathcal{A}(n)$ satisfying the above condition for all hyperplanes $\Pi$ through the origin, can we conclude that there exists a transformation $T \in \Omega$ such that $T(K_2) = K_1$?

A particularly interesting case arises when $K_1 = K_2$ and $\Omega = O(n)$, the orthogonal group. In this situation, it is natural to conjecture that $K_1$ is a body of revolution.

\textbf{Question 2.} Let $K \subset \mathbb{R}^n$ be a convex body with $n \geq 3$, and let $p \in \mathbb{R}^n$. Suppose that every $(n-1)$-dimensional section of $K$ through $p$ has some geometric property $X$. For which properties $X$ can we conclude that $K$ itself possesses property $X$?

This second question is classical in the area of geometric tomography. In \cite{Rogers1965}, C.~A. Rogers made contributions to both questions. For Question 1, he showed that if $\Omega$ is the group of homotheties, then there exists a transformation $T \in \Omega$ such that $T(K_2) = K_1$. As a consequence, for Question 2, he proved that if property $X$ is central symmetry, then $K$ must be centrally symmetric.

Moreover, Rogers conjectured that if not all sections of $K$ through $p$ are centered at $p$, then the assumption of central symmetry characterizes ellipsoids. Related works include \cite{Aitchison1971}, \cite{Larman1974}, \cite{MontejanoMorales2003}, \cite{MontejanoMorales2007a}, \cite{MontejanoMorales2007b}.

The principal result of this paper provides a generalization of Question 1. Instead of examining pairs of parallel hyperplanes through two fixed points, we study pairs of hyperplanes that pass through two fixed points and a common $(n-2)$-dimensional plane contained in a given hyperplane.

We now introduce the necessary definitions.

Let $H \subset \mathbb{R}^n$ be a hyperplane. A mapping $S : \mathbb{R}^n \to \mathbb{R}^n$ is a \emph{reflection} with respect to $H$ (a mirror) if, for every point $x \in \mathbb{R}^n$, the point $S(x)$ lies on the line orthogonal to $H$ through $x$, at equal distance from $H$, and on the opposite side of $H$ from $x$. A convex body $K \subset \mathbb{R}^n$ is said to be \emph{symmetric} with respect to $S$ if $S(K) = K$.

Let $p_1, p_2 \in \mathbb{R}^n \setminus H$ be distinct points, and let $M \subset H$ be an $(n-2)$-plane. Define $\pi_1(M)$ and $\pi_2(M)$ as the affine hyperplanes generated by $M$ and $p_1$, and by $M$ and $p_2$, respectively.

\begin{theorem}\label{mainthm}
	Let $K_1, K_2 \subset \mathbb{R}^n$, $n \geq 3$, be strictly convex bodies, and let $H \subset \mathbb{R}^n$ be a hyperplane. Let $p_1, p_2 \in \mathbb{R}^n \setminus H$ be distinct. Suppose that for each $(n-2)$-plane $M \subset H$, there exists a reflection $S_M : \mathbb{R}^n \to \mathbb{R}^n$ with respect to a hyperplane through $M$ such that
	\[
	S_M\big(\pi_2(M) \cap K_2\big) = \pi_1(M) \cap K_1.
	\]
	Then there exists a reflection $S : \mathbb{R}^n \to \mathbb{R}^n$ with respect to $H$ such that
	\[
	S(K_2) = K_1.
	\]
         If $S(p_2)\not = p_1$, then $K_1,K_2$ are bodies of revolution.
\end{theorem}

\textbf{Structure of the paper.} Let $L$ be the line joining $p_1$ and $p_2$. The paper is organized as follows:
\begin{enumerate}
	\item Introduction (this section).
	\item Reduction of the general case to dimension 3.
	\item Auxiliary result: A characterization of the circle.
	\item Lemmas for $n = 3$ and $p_1 \notin K_1$.
	\item Lemmas for $n = 3$, $p_1 \in \mathrm{int}(K_1)$, and $L \not\perp H$.
	\item Lemmas for $n = 3$, $p_1 \in \mathrm{int}(K_1)$, and $L \perp H$.
	\item Proof of Theorem~\ref{mainthm} for $n = 3$.
\end{enumerate}

	\section{Reduction of the General Case to Dimension 3}
	
	Let $\mathbb{R}^n$ be the $n$-dimensional Euclidean space endowed with the standard inner product $\langle \cdot, \cdot \rangle : \mathbb{R}^n \times \mathbb{R}^n \to \mathbb{R}$. We choose an orthonormal coordinate system $(x_1, \dots, x_n)$ for $\mathbb{R}^n$. Define the closed unit $n$-ball by
	\[
	B(n) = \{x \in \mathbb{R}^n : \|x\| \leq 1\}, \quad S^{n-1} = \{x \in \mathbb{R}^n : \|x\| = 1\}
	\]
	as its boundary.
	
	Let $x, y \in \mathbb{R}^n$. Denote by $L(x, y)$ the line through $x$ and $y$, and by $[x, y]$ the line segment connecting them. For sets $A, B \subset \mathbb{R}^n$, let $\mathrm{aff}\{A, B\}$ be the affine hull of $A \cup B$. For any hyperplane $\Gamma \subset \mathbb{R}^n$, let $\psi_\Gamma : \mathbb{R}^n \to \Gamma$ denote the orthogonal projection onto $\Gamma$.
	
	\begin{lemma}
		Let $K_1, K_2 \subset \mathbb{R}^n$ be convex bodies, $n \geq 3$, and let $H \subset \mathbb{R}^n$ be a hyperplane. Suppose that for every hyperplane $\Gamma$ orthogonal to $H$, there exists a reflection $s : \Gamma \to \Gamma$ with respect to $H \cap \Gamma$ such that
		\[
		s(\psi_\Gamma(K_2)) = \psi_\Gamma(K_1).
		\]
		Then the reflection $S_H : \mathbb{R}^n \to \mathbb{R}^n$ with respect to $H$ satisfies
		\[
		S_H(K_2) = K_1.
		\]
	\end{lemma}
	
	\begin{proof}
		By hypothesis, for every hyperplane $\Gamma$ orthogonal to $H$, we have
		\[
		\psi_\Gamma(K_1) = s(\psi_\Gamma(K_2)) = \psi_\Gamma(S_H(K_2)) = \psi_\Gamma(W_1),
		\]
		where $W_1 := S_H(K_2)$. Therefore, $\psi_\Gamma(K_1) = \psi_\Gamma(W_1)$ for all $\Gamma$ orthogonal to $H$, and it follows that $K_1 = W_1 = S_H(K_2)$.
	\end{proof}
	
	\begin{corollary}
		Let $K \subset \mathbb{R}^n$ be a convex body and $H \subset \mathbb{R}^n$ a hyperplane. If for every hyperplane $\Gamma$ orthogonal to $H$ the projection $\psi_\Gamma(K)$ is symmetric with respect to $H \cap \Gamma$, then $K$ is symmetric with respect to $H$.
	\end{corollary}
	
	\begin{lemma}
		Under the hypotheses of Theorem \ref{mainthm}, let $\Gamma$ be a hyperplane orthogonal to $H$. Then the sets $\psi_\Gamma(K_1)$, $\psi_\Gamma(K_2)$, $\psi_\Gamma(p_1)$, $\psi_\Gamma(p_2)$, and $\psi_\Gamma(H)$ satisfy the conditions of Theorem~1 in dimension $n - 1$.
	\end{lemma}
	
	\begin{proof}
		Assume $K_1, K_2, p_1, p_2$ and $H$ satisfy the conditions of Theorem~1. Let $M \subset \psi_\Gamma(H)$ be a $(n-3)$-dimensional plane. Choose $(n-2)$-planes $W_1, W_2 \subset \Gamma$ such that $M \subset W_i$ and $p_i \in W_i$ for $i = 1, 2$. By hypothesis, there exists a reflection $S : \mathbb{R}^n \to \mathbb{R}^n$ with respect to a hyperplane $\Omega$ containing $\psi_\Gamma^{-1}(M) \subset H$, such that
		\[
		S(\Pi_2 \cap K_2) = \Pi_1 \cap K_1,
		\]
		where $\Pi_i := \psi_\Gamma^{-1}(W_i)$ for $i = 1, 2$.
		
		Let $y \in W_2 \cap \psi_\Gamma(K_2) = \psi_\Gamma(\Pi_2) \cap \psi_\Gamma(K_2) = \psi_\Gamma(\Pi_2 \cap K_2)$. Define $\bar{y} := s(y)$ to be the image of $y$ under the reflection $s : \Gamma \to \Gamma$ with respect to $\psi_\Gamma(\Omega)$. We claim that
		\[
		\bar{y} \in W_1 \cap \psi_\Gamma(K_1) = \psi_\Gamma(\Pi_1) \cap \psi_\Gamma(K_1) = \psi_\Gamma(\Pi_1 \cap K_1).
		\]
		
		Since $y = \psi_\Gamma(x)$ for some $x \in \Pi_2 \cap K_2$, and $S(x) \in \Pi_1 \cap K_1$ by hypothesis, then $\bar{y} = \psi_\Gamma(S(x))$ lies in $\psi_\Gamma(\Pi_1 \cap K_1)$, completing the proof.
	\end{proof}
	
	\textbf{Inductive Step.} Assume Theorem \ref{mainthm} holds in dimension $n \geq 3$. Let $K_1, K_2 \subset \mathbb{R}^{n+1}$, $H \subset \mathbb{R}^{n+1}$, and $p_1, p_2 \in \mathbb{R}^{n+1} \setminus H$ satisfy the conditions of Theorem~1. Let $\Gamma$ be a hyperplane orthogonal to $H$. Then $\psi_\Gamma(K_1)$, $\psi_\Gamma(K_2)$, $\psi_\Gamma(p_1)$, $\psi_\Gamma(p_2)$, and $\psi_\Gamma(H)$ satisfy the conditions of Theorem~1 in dimension $n$. By the inductive hypothesis, there exists a reflection $S : \Gamma \to \Gamma$ with respect to $\psi_\Gamma(H)$ such that $S(\psi_\Gamma(K_2)) = \psi_\Gamma(K_1)$. By Lemma 1, this implies that $S_H(K_2) = K_1$ in $\mathbb{R}^{n+1}$.
	\qed
	
	
	\section{A Characterization of the Circle}
	
	To prove Theorem~1, we need a characterization of the circle, which we can informally state as follows:
	
	\textit{A convex domain in the plane, in which all inscribed rectangles have a pair of opposite sides passing through two fixed points, must be a circle.}
	
	This result will be formally presented in Lemma~3 below. Its proof uses a well-known characterization of central symmetry due to A. Rogers~\cite{Rogers1981}. Let $\text{length}(W)$ denote the length of the segment $W$.
	
	\begin{lemma}[Rogers]
		Let $M \subset \mathbb{R}^2$ be a convex domain and $a, b \in \operatorname{int}(M)$. Suppose that for every pair of parallel lines $A, B$ passing through $a$ and $b$, respectively, the equality
		\[
		\text{length}(A \cap M) = \text{length}(B \cap M)
		\]
		holds. Then $M$ is centrally symmetric, with center at the midpoint of the segment $[a, b]$.
	\end{lemma}
	
	We now provide the geometric construction that motivates the main result.
	
	In $\mathbb{R}^2$, consider a circle $C$ centered at the origin with radius $r$. Fix the point $(t, 0)$, where $0 < t < r$. For every unit vector $u$, let $M_u$ and $-M_u$ be the lines parallel to $u$ and passing through $(t, 0)$ and $(-t, 0)$, respectively. These lines define the vertical sides of a rectangle inscribed in $C$. Denote by $N_u$ and $-N_u$ the orthogonal lines (perpendicular to $u$) that form the other pair of rectangle sides. For each $u \in S^1$, define $\operatorname{l_{rt}}(u)$ to be the length of the chord $M_u \cap C$, and $\operatorname{d_{rt}}(u)$to be the distance between $N_u$ and the origin.
	
	Since the line orthogonal to $M_u$ and passing through the midpoint of $M_u \cap C$ goes through the origin, we have
	\[
	\operatorname{d_{rt}}(u) = \frac{1}{2} \operatorname{l_{rt}}(u).
	\]
	Thus, $\operatorname{d_{rt}}(u)$ is a non-zero, even, and continuous function on the unit circle $S^1$.
	
	\begin{lemma}
		Let $M \subset \mathbb{R}^2$ be a convex domain and let $a, b \in \operatorname{int}(M)$. Suppose that for every pair of parallel lines $A, B$ passing through $a$ and $b$, respectively, the chords $A \cap M$ and $B \cap M$ define a rectangle inscribed in the boundary $\partial M$. Then $M$ is a circle.
	\end{lemma}
	
	\begin{proof}
		By Rogers' Theorem, $M$ is centrally symmetric about the midpoint of $[a, b]$.
		
		Take a pair of parallel chords $[c_0, d_0]$ and $[c_0', d_0']$ of $M$ passing through $a$ and $b$, respectively. By hypothesis, the quadrilateral $c_0 c_0' d_0 d_0'$ is a rectangle. Let $D_0$ be its circumcircle.
		
		Let $c, d$ be the endpoints of the chord $L(a, b) \cap \partial M$. Choose $[c_0, d_0], [c_0', d_0']$ so that $c_0 \neq c \neq c_0'$ and $d_0 \neq d \neq d_0'$. Then, the line $L(a, d_0')$ must intersect $D_0$ at another point, say $c_1$. Similarly, the line $L(b, c_0)$ intersects $D_0$ at another point $d_1'$. Renaming $d_0'$ as $d_1$ and $c_0$ as $c_1'$, we obtain a new rectangle $c_1 d_1 c_1' d_1'$.
		
		We will show that $c_1 \in \partial M$ and hence $d_1' \in \partial M$ as well. Since $L(a, d_0')$ intersects $\partial M$ at some point $t$, and this point also lies on the line orthogonal to $L(b, c_0)$ passing through $c_0$, it follows that $t = c_1$. A similar argument confirms $d_1' \in \partial M$.
		
		Iterating this construction, we generate a sequence of rectangles inscribed in both $D_0$ and $M$. In particular, we obtain sequences $\{c_n\}, \{d_n\} \subset D_0 \cap \partial M$ such that $c_n \to c$ and $d_n \to d$ as $n \to \infty$. Hence, $c$ and $d$ lie on $D_0$.
		
		It follows that $D_0$ is the circle centered at the midpoint of $[a, b]$ with diameter $[c, d]$. Because the chords $[c_0, d_0]$ and $[c_0', d_0']$ were arbitrary, we conclude that $M = D_0$. Therefore, $M$ is a circle.
	\end{proof}
	
	\begin{figure}
    \centering
    \includegraphics[width=.88\textwidth]{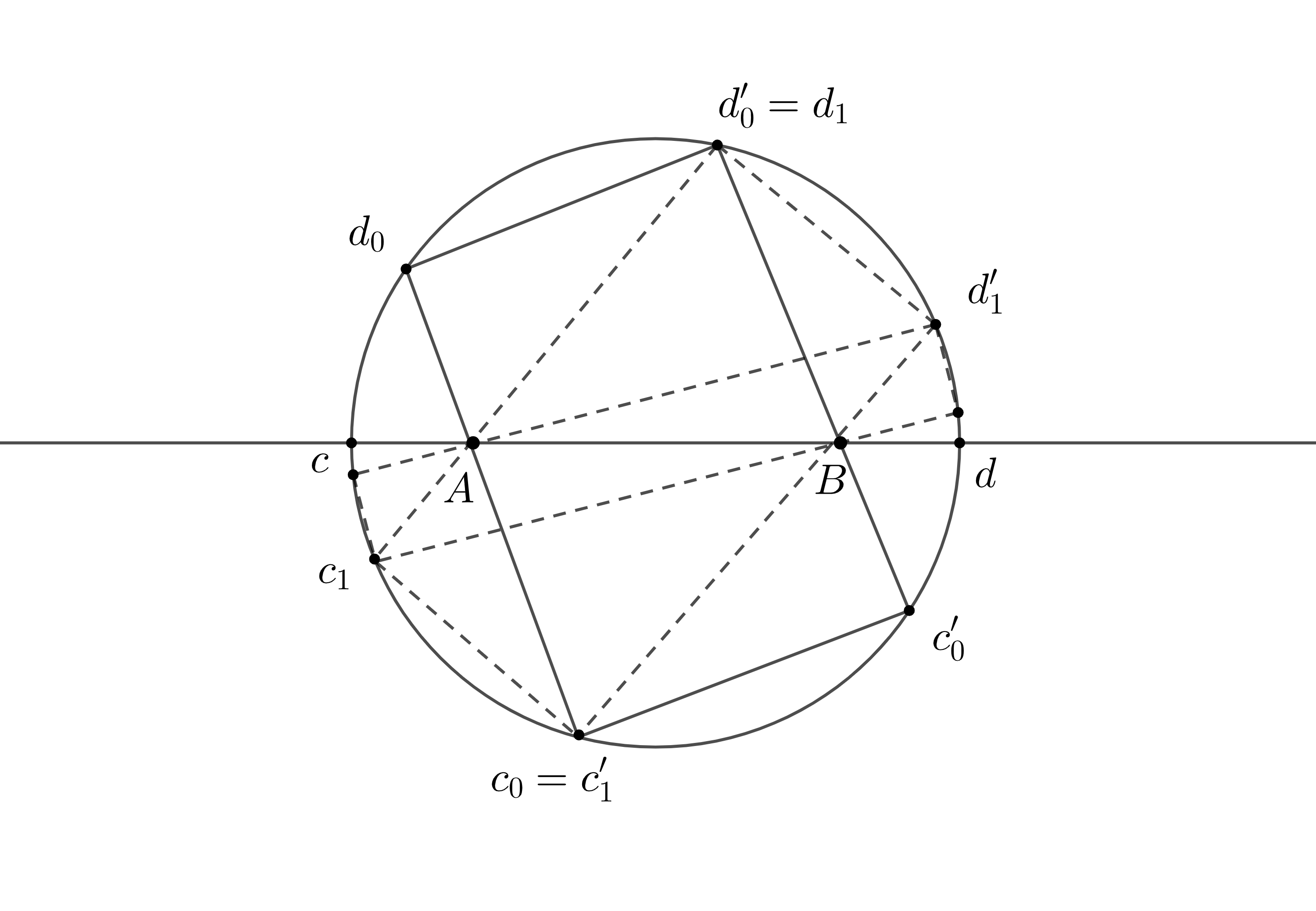}
    \caption{Convex domain with inscribed rectangles with pair of opposite sides passing through two fix points is a circle.}
    \label{santo}
\end{figure}
	
	\section{Lemmas for the Case $n = 3$ and $p_1 \notin K_1$}
	
	Assume that the convex bodies $K_1, K_2 \subset \mathbb{R}^3$, the points $p_1, p_2 \in \mathbb{R}^3 \setminus H$, and the plane $H \subset \mathbb{R}^3$ satisfy the hypotheses of Theorem~1. Let $L := L(p_1, p_2)$ be the line through $p_1$ and $p_2$.
	
	Suppose $L \cap H \neq \emptyset$ (otherwise, if $H$ and $L$ are parallel, we may replace $K_2$ with its reflection $S(K_2)$ with respect to $H$, so that the new line $\bar{L} := L(p_1, S(p_2))$ intersects $H$).
	
	Let $\Sigma := L \cap H$ and $S$ be the reflection with respect to the plane $H$. Choose a coordinate system such that $\Sigma$ is at the origin and $H$ is given by $z = 0$, with $p_1$ having positive $z$-coordinate and $p_2$ negative.
	
	\begin{lemma}\label{orto}
		If $p_1 \notin K_1$, then $p_2 \notin K_2$ and the line $L$ is perpendicular to $H$.
	\end{lemma}
	
	\begin{proof}
		We first claim that $p_2 \notin \operatorname{int}(K_2)$. Let $G_1$ be the projection of $K_1$ from $p_1$ onto $H$. Let $M$ be a supporting line of $G_1$ touching it at a unique point. Then $\pi_1(M) \cap K_1$ is at most 1-dimensional. If $p_2 \in \operatorname{int}(K_2)$, then $\pi_2(M) \cap K_2$ would be 2-dimensional, contradicting the existence of a reflection mapping $\pi_2(M) \cap K_2$ onto $\pi_1(M) \cap K_1$. Thus, $p_2 \notin \operatorname{int}(K_2)$.
		
		Let $G_2$ be the projection of $K_2$ from $p_2$ onto $H$. We show $G_1 = G_2$. Assume otherwise, say $x \in G_1 \setminus G_2$. Without loss of generality, let $x \in \operatorname{int}(G_1)$. Let $W \subset H$ be a line separating $x$ from $G_2$ and let $M$ be a line parallel to $W$ through $x$. Since $x \in \operatorname{int}(G_1)$, $\pi_1(M) \cap K_1$ contains interior points. But $M \cap G_2 = \emptyset$, so $\pi_2(M) \cap K_2 = \emptyset$, a contradiction. Hence, $G_1 = G_2$.
		
		Suppose $L$ is not perpendicular to $H$. Let $\alpha, \beta$ be the projections of $p_1, p_2$ onto $H$, respectively. Since $L$ is not perpendicular to $H$, $\alpha \neq \beta$.
		
		We claim there exists a point $A \in \partial G_1$ and two perpendicular lines $M_A, N_A$ such that $M_A$ is a supporting line to $G_1$ with $M_A \cap \partial G_1 = \{A\}$, $N_A$ passes through $A$, and $N_A \cap \operatorname{int}[\alpha, \beta] \neq \emptyset$.
		
		Let $w \in \operatorname{int}[\alpha, \beta]$. Choose a disc $B \subset \mathbb{R}^2$ centered at $w$ such that $G_1 \subset \operatorname{int}(B)$ (this is possible since $G_1$ is compact). Shrink the radius of $B$ until $B \cap \partial G_1 \neq \emptyset$. Let $A \in B \cap \partial G_1$, and let $M_A$ be a supporting line of $B$ through $A$. Since $G_1 \subset B$, $M_A$ is also a supporting line of $G_1$. The tangency condition implies $M_A \cap \partial G_1 = \{A\}$.
		
		Let $N_A := L(A, w)$, and let $\Delta$ be the plane perpendicular to $H$ containing $N_A$, with open half-spaces $\Delta^+, \Delta^-$ such that $p_1 \in \Delta^+$ and $p_2 \in \Delta^-$. Define $\Gamma_1 := \operatorname{aff}\{M_A, p_1\}$, $\Gamma_2 := \operatorname{aff}\{M_A, p_2\}$, and let $S_M$ be the reflection with $S_M(\Gamma_2 \cap K_2) = \Gamma_1 \cap K_1$.
		
		Since $L(p_2, A)$ is a supporting line of $K_2$, it intersects $K_2$ at some $\tau \in \partial K_2 \cap \Delta^-$. Then $S_M(\tau) \in \partial K_1 \cap \Gamma_1$. Thus, $L(p_1, S_M(\tau))$ is a supporting line of $K_1$, but it cannot pass through $A$, contradicting the choice of $M_A$. Therefore, $L$ must be perpendicular to $H$.
	\end{proof}
	
	\textbf{Definition.} Let $K \subset \mathbb{R}^n$ be a convex body and $x \in \mathbb{R}^n \setminus K$. The union of all tangent lines of $K$ passing through $x$ is called the \emph{support cone} of $K$ with respect to $x$, denoted $C(K, x)$.
	
	\begin{lemma}\label{pastel}
		Suppose $p_1 \notin K_1$. Then $\|p_1 - \Sigma\| = \|p_2 - \Sigma\|$.
	\end{lemma}
	
	\begin{proof}
		Assume $\|p_1 - \Sigma\| \neq \|p_2 - \Sigma\|$. From Lemma 4, $p_2 \notin K_2$ and $L$ is perpendicular to $H$. Let $M \subset H$ be a line through $\Sigma$. Then $\pi_1(M) = \pi_2(M) =: \Pi$, and the reflection $S_M$ becomes $S$, the reflection across $H$.
		
		The reflection $S$ maps the support cone $C(K_2, p_2)$ onto $C(K_1, p_1)$. That is, $S(p_2) = p_1$. But this contradicts our assumption that $\|p_1 - \Sigma\| \neq \|p_2 - \Sigma\|$. Hence, $\|p_1 - \Sigma\| = \|p_2 - \Sigma\|$.
	\end{proof}
	
	\begin{lemma}\label{rica}
		If $S(p_2) = p_1$, then $S(K_2) = K_1$.
	\end{lemma}
	
	\begin{proof}
		Since $S(p_2) = p_1$, the line $L$ is perpendicular to $H$. For every line $M \subset H$ through $\Sigma$, the planes $\pi_1(M)$ and $\pi_2(M)$ coincide. Therefore, the restriction of the global reflection $S$ to $\pi_1(M)$ is simply the reflection across the line $M$.
		
		Thus, for each plane $\Pi$ through $L$, the section $\Pi \cap K_1$ is the reflection (across $H \cap \Pi$) of $\Pi \cap K_2$. Therefore, $S(K_2) = K_1$.
	\end{proof}
	
	\section{Lemmas for the Case $n = 3$, $p_1 \in \operatorname{int}(K_1)$, and $L \not\perp H$}
	
We assume that the convex body $K$, the points $p_1$, $p_2$ and the plane $H$ satisfies the condition of Theorem~1 for dimension 3. Let $L:=L(p_{1},p_{2})$. Suppose  that $p_{1}\in \operatorname{int}( K_{1})$, $L\cap H \not= \emptyset$ and $L$ is not orthogonal to $H$. We denote by $\Sigma$ the point $L\cap H$ and by $S$ the reflection with respect to the plane $H$. We take a system of coordinates such that $\Sigma$ is the origin and $H$ is given by the equation $z=0$. We choose the notation such that $p_{1}$ has positive coordinate $z$ and $p_{2}$ has negative coordinate $z$. 

Since $p_{1}\in \operatorname{int}(K_{1})$ is clear that $p_{2}\in \operatorname{int}( K_{2})$. Let $\Delta$ be the plane defined by $L$ and the unit normal vector of $H$, say $n$, and let $H_{i}$ be plane parallel to $H$ and passing through $p_{i}$, $i=1,2$. Let $N_1:=H_1 \cap K_1$, $N_2:=H_2 \cap K_2$.

The corresponding reflection between $N_{1}$ and $N_{2}$ is a translation defined by a vector $u$ parallel to $n$, i.e., $N_1=N_2+u$.
\begin{lemma}\label{dey}
The equality 
\begin{eqnarray}\label{trinidad}
\| p_1- \Sigma \|= \| p_{2}- \Sigma \|
\end{eqnarray}
holds.
\end{lemma}
\begin{proof}
On the contrary to Lemma's claim, we suppose that $\|p_1-\Sigma\| \not= \|p_{2}-\Sigma\|$, say $\|p_1-\Sigma\| > \|p_{2}-\Sigma\|$. Let $L_1 \subset H_1$, $L_2 \subset H_2$ be parallel lines passing through $p_1$ and $p_2$, respectively. First, we assume that 
\begin{eqnarray}\label{mango}
\len(L_2 \cap N_2)  > \len(L_1 \cap  N_1).
\end{eqnarray}
We denote by $C_1$ and $C_2$ the collections of all chords of $K_1$ and $K_2$, respectively, parallel to $L_2$ and with length equal to $L_2 \cap N_2$. Let $U$ and $V$ be orthogonal planes to $L_2 \cap N_2$ and passing through the extreme points of this line segment. The sets $C_1$ and $C_2$ are contained in the strip determined by $U$ and $V$. Taking lines $W\subset H$, parallel to $L_1$, we can see that the chords of $C_2$ are mapped into the chords of $C_1$ by the reflection which map the plane $\pi_2(W)$ into the plane $\pi_1(W)$. In the same way, the points of $ U \cap K_2$ are mapped into the points of $ U \cap K_1$ and the points of $V \cap K_2$ are mapped into the points of $V \cap K_1$ (see Fig. \ref{demon}). By virtue of (\ref{mango}) it follows that $L_1 \cap  N_1$ does not belong to $C_1$, i.e., the point $z_1:=L_1\cap U $ does not belong to the convex figure $U \cap K_1$, however, $z_2:=L_2\cap U$ belongs to $\bd (U \cap K_2)$. Let $M\subset U $ be a line which separate $U \cap K_1$ from $z_1$. We denote by $M'$ the line parallel to $M$ passing through $z_1$. Let $\Pi$ be the plane $\aff \{L_1,M'\}$ and let $N:=\Pi \cap H$. We have $\Pi \cap C_1=\emptyset$ and $\pi_2(N)\cap C_2 \not= \emptyset.$ We consider the reflection $S_N:\pi_2(N) \rightarrow \Pi:=\pi_1(N)$ with respect to a plane passing through $N$ which maps $\pi_2(N)\cap K_2$ into $\pi_1(N)\cap K_1$ given by the hypothesis, we have that $S_N(L_2\cap  N_2)$ is a chord of $C_1$ but this is in contradiction with the fact $\Pi \cap C_1=\emptyset$.
\begin{figure}
    \centering
    \includegraphics[width=.78\textwidth]{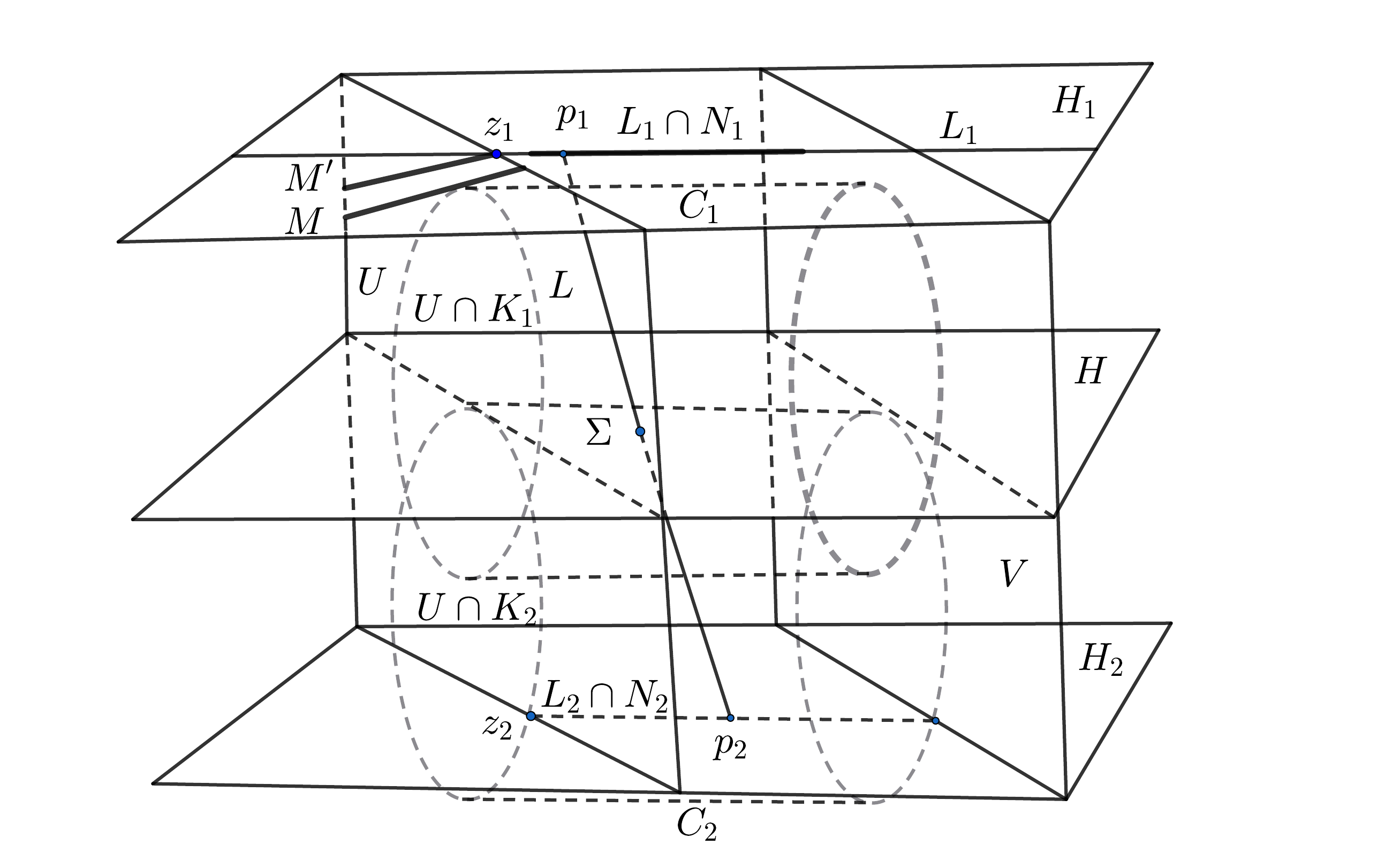}
    \caption{For all pairs of parallel lines $L_i \subset H_i$, $p_i\in L_i$, $i=1,2$, we have $\len(L_1 \cap  N_1)=\len(L_2 \cap  N_2)$.}
 \label{demon}
\end{figure}
Then we can assume that for all pairs of parallel lines $L_i \subset H_i$, $p_i\in L_i$, $i=1,2$, we have  
\begin{eqnarray}\label{campesina}
\len(L_1 \cap  N_1)=\len(L_2 \cap  N_2)
\end{eqnarray}
By assumption $L$ is not orthogonal to $H$, then $p_1\not= p_2+ u$ and, moreover, since $N_1=N_2+u$ we have
\[
\len(L_2\cap N_2) =\len[(L_2\cap N_2)+u] = \len[(L_2+u)\cap N_1].
\]
By (\ref{campesina}),
\[
\len(L_1\cap N_1) =\len[(L_2+u)\cap N_1],
\]
i.e., the pairs of parallel chords of $N_1$, $ L_1\cap N_1 $ and $(L_2+u) \cap N_1$ passing through $p_1$ and $(p_2+u)$, respectively, have the same length and the extreme points of this segments are vertices of an scribed rectangle in $N_1$, notice that the segment $L_1 \cap N_1$, $L_2 \cap N_2$ and $(L_2+u)\cap N_1$ defined an orthogonal prism (see Fig \ref{panter}). Thus the convex figure $N_1$ and the different points $p_1$ and $p_2+u$ satisfies the condition of Lemma \ref{santo} and, consequently, $N_1$ is a circle. Since $N_2=(-u)+N_1$, $N_2$ is a circle also.
\begin{figure}
    \centering
    \includegraphics[width=.78\textwidth]{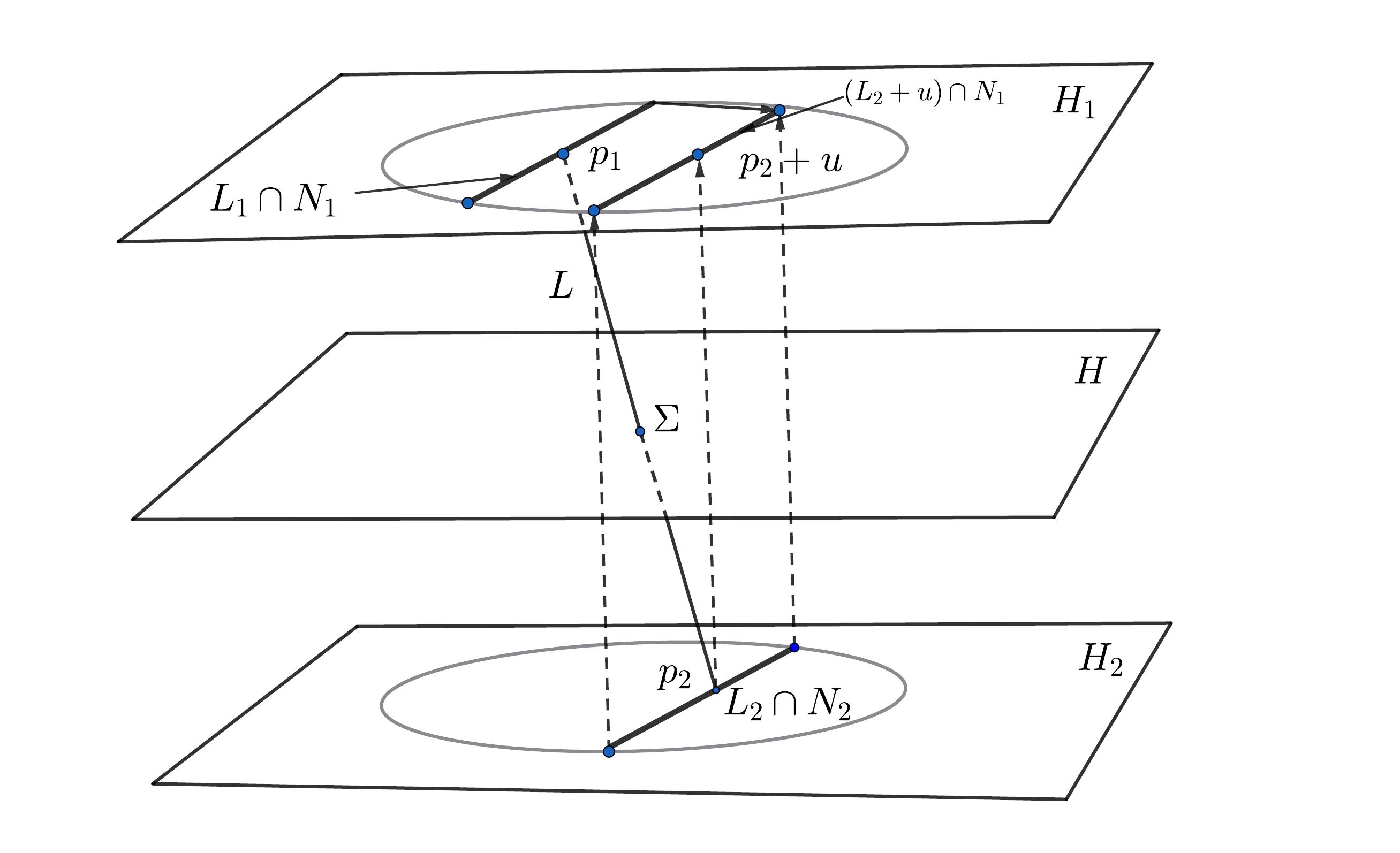}
    \caption{The sections $N_1$, $N_2$ are circles.}
 \label{panter}
\end{figure}
We are going to prove that the sections $N'_1=H'_1 \cap K_2$ and $N'_2=H'_2 \cap K_1$  are circles, where $H'_i=S(H_i)$, $i=1,2$. In order to prove this, we will use the Lemma \ref{santo}. We denote by $c$ and $r$ the center and the ratio of $N_1$, respectively, and let $W$ be a line orthogonal to $H$ and passing through $c$. We take a coordinate system such that the point $H \cap W$ is the origin, $H$ is the plane $z=0$ and the plane $\Delta$ has equation $y=0$. On the other hand, let $ p'_i \in L$ 
such that $\| p'_i-\Sigma\|=\| p_i-\Sigma\|$, $i=1,2$, and for all line $L_i\subset H_i, p_i \in L_i$, we denote by $ L'_i\subset H'_i$ the line parallel to $L_i$  and passing through $ p'_i$, $i=1,2$. 

We observe that for every plane $\Pi$, $L \subset \Pi$, the section $\Pi \cap K_1$ is a reflection of the section $\Pi \cap K_2$ with respect to $\Pi \cap H$. Notice that for every line $M\subset H$, $M$  is invariant under the reflection $S_M:\Rt \rightarrow \Rt$ given by the hypothesis of Theorem~1, in particular, when $M$ is passing through $\Sigma$, we have that $M=\Pi\cap H$, 
$\pi_1(M)=\Pi=\pi_2(M)$, $S_M$ is a reflection with respect to a plane perpendicular to $\Pi$, containing $M$ and $S(M)=M$. Consequently, the restriction of $S_M$ to $\Pi$ is the reflection with respect to the line $M$. Hence for all line $L_1\subset H_1, p_1 \in L_1,$ the chords $L_1 \cap N_1$ and $ L'_1 \cap N'_1 \in C_2$ have the same length. As we already have seen, the chord $[e,f]$ of $ N'_1$ parallel and with the same length than $ L'_1 \cap N'_1$ belongs to $C_2$. Since $U$ and $V$ are orthogonal to the chords of $C_2$, the quadrilateral inscribed in $ N'_1$ limited by the chords $ L'_1 \cap  N'_1$, $[e,f]$, $U \cap  N'_1$ and $V \cap  N'_1$ is a rectangle (see Fig. \ref{wagner}).  Thus $N'_1$ satisfies conditions of Lemma \ref{santo} with respect to point $ p'_1$. Hence $N'_1$ is a circle. 
\begin{figure}
    \centering
    \includegraphics[width=.78\textwidth]{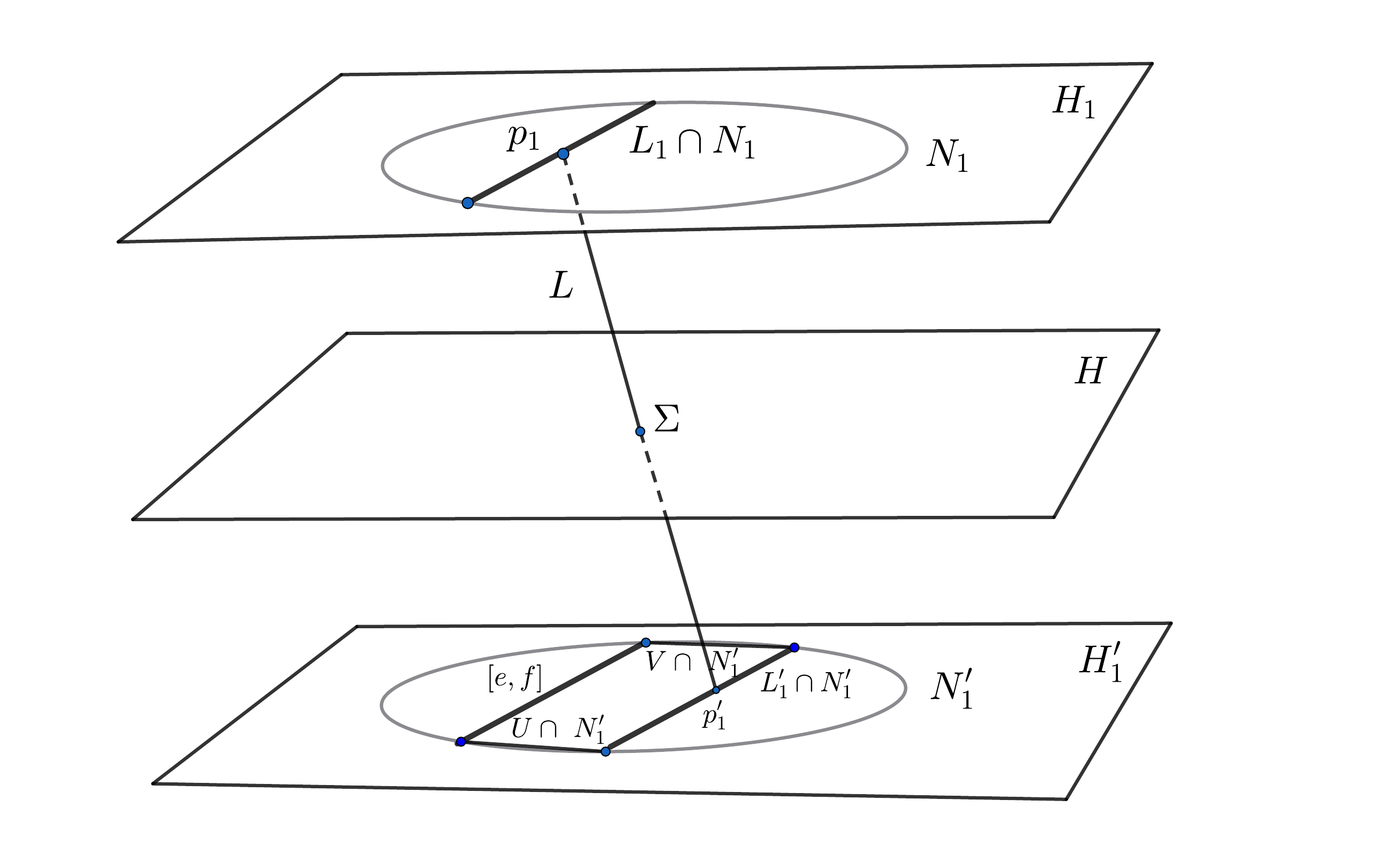}
    \caption{The sections $N'_1$, $N'_2$ are circles.}
 \label{wagner}
\end{figure}
According with the notation given above and in the Lemma \ref{santo}, since $N_1$ is a circle, for each line $L_1$, $p_1\in L_1$, parallel to a unit vector $u\in \Suno\subset H$, the planes $U$ and $V$ are at a distance $d_{rt}(u)$ from $c$, where $t$ is the $x$ coordinate of $p_1$. Thus all the chords of $N_1$ passing through $p_1$ have the same length than the corresponding parallel chords  of $N'_1$ passing through $p'_1$ (such length is given by the function $2d_{rt}(u)$), hence $N_1$ and $N'_1$ are circles of the same radius. 

By analogous arguments, $N'_2$ is a circle of the same radius than $N_2$. Thus the four circles $N_1$, $N_2$, 
$N'_1$ and $N'_2$ have the same ratio.

Now we take a system of coordinates such that $\Sigma$ is the origin and $ H$ is the plane $z=0$ and $\Delta$ is plane $y=0$. Since we are assuming that $L$ is not orthogonal to $H$, we can assume that $p_1$ has positive $x$ coordinate and $p_2$ has negative $x$ coordinate. Let $q\in H'_2 \cap  \Delta $ be a point with the same coordinate $x$ as $p_2$, $d$ be the point determined by the intersection of the lines $L(p_1, q)$ and the axis $x$ (See Fig. \ref{guitarra}). Let $Q,D,L_2$ be three lines orthogonal to $\Delta$ passing through $q$, $d$ and $p_2$, respectively. We denote by $\pi_{1}(D)$ and $\pi_{2}(D)$ the planes $\aff \{D,p_1\}$ and $\aff \{D,p_2\}$, respectively and we consider the reflection $S_D:  \Rt \rightarrow \Rt$ which 
satisfies $S_D(\pi_{2}(D) \cap K_2)= \pi_{1}(D) \cap K_1$. By virtue that
the plane $H$ bisect the angle between the planes $\pi_{1}(D)$ and $\pi_{2}(D)$ and since the triangle $d p_2 q$ is isosceles, it follows that $S_D(L_2)=Q$. Thus 
\[
\len(L_2 \cap N_2) =\len(Q \cap N'_2).
\]
On the other hand, we know that 
\[
\len( L_2 \cap N_2)=\len(L'_2 \cap N'_2),
\]
then
\[
\len(Q\cap N'_2) =\len( L'_2\cap N'_2).
\]
Hence the center of $N'_2$ is in the $z$ axis. Thus $N_2$ has its center in the $z$ axis. However, since $\|p_2-\Sigma\|< \|p_1 -\Sigma\|$, the absolute value of the $x$ coordinate of $p_1$ is $\geq$ than the absolute value $x$ coordinate of $p'_2$ and the mid point of the segment $[p_1, p_2 +u]$ can not be in the axis $z$, but such mid point is the center $c$ of $N_1$. This is in contradiction with  the fact $N_1= u+ N_2$, $u$ parallel to $z$. From here the equality (\ref{trinidad}) holds and the Lemma follows.
\end{proof}
\begin{figure}
    \centering
    \includegraphics[width=.78\textwidth]{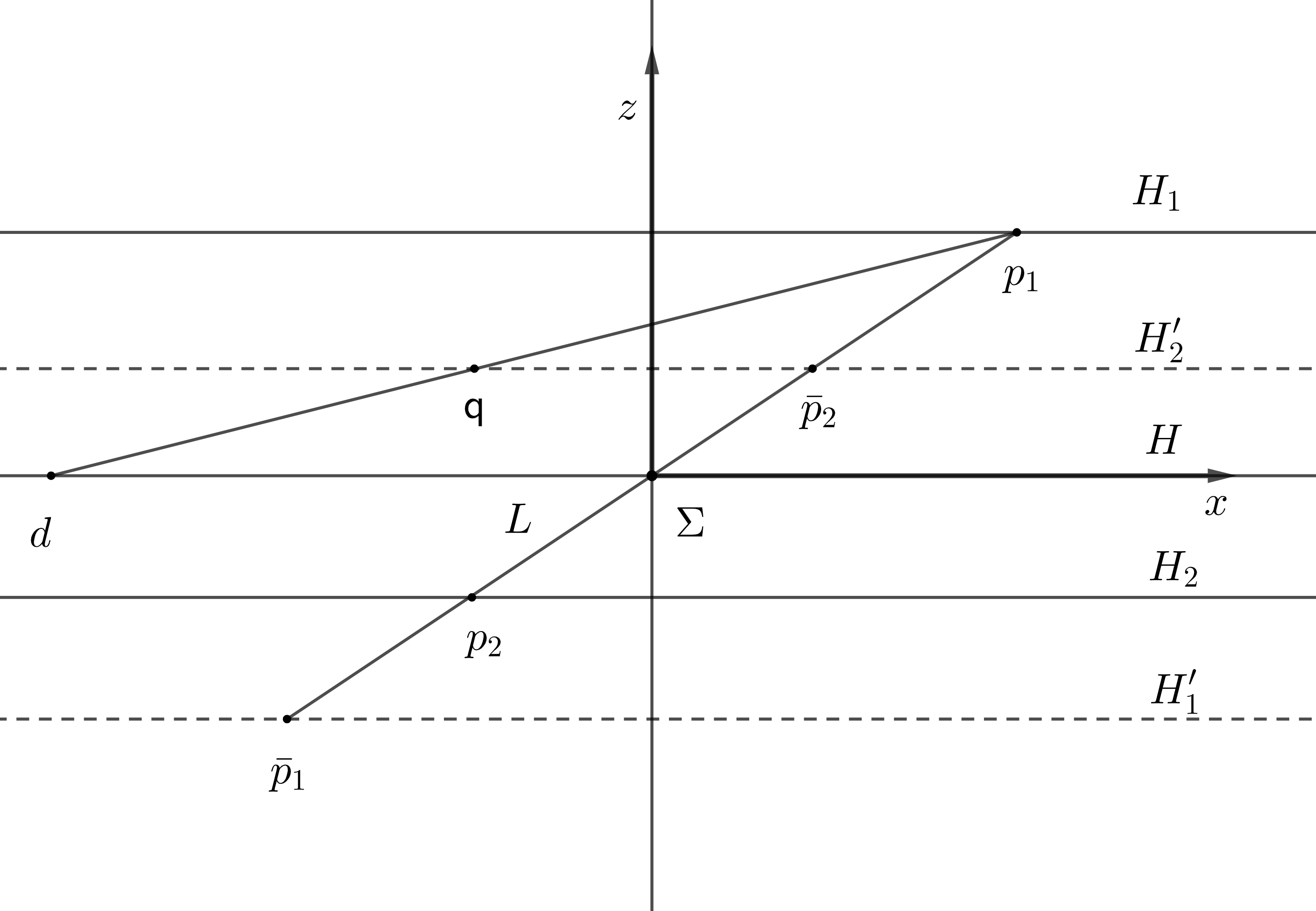}
    \caption{The center of $N'_2$ is in the $z$ axis.}
    \label{guitarra}
\end{figure}
\begin{lemma}\label{isabella} 
$K_1$ is the image of $K_2$ under the reflection with respect to plane $H$, i.e., 
\begin{eqnarray}\label{magia}
S(K_2)=K_1.
\end{eqnarray}
\end{lemma}
\begin{proof}
Let $M_1 :=L(p_1,S(p_2))$ and we denote by $M_2$ the line parallel to $M_1$ passing through $ p_ 2$. We observe that, by Lemma \ref{dey}, $M_1$ is parallel to $H$. On the other hand, notice that, for a line $M\subset H \backslash \Delta$ parallel  to $M_1$, we have $S_M=S$. Since (\ref{trinidad}) of Lemma \ref{dey} holds, $M_1$ and $M_2$ are at the same distance from $H$. This implies the angle between $\pi_1 (M)$ and $\pi_2 (M)$  is bisected by $H$. Thus 
\[
S=S_M \textrm{ }\textrm{ for all line } M\subset H \backslash \Delta \textrm{ }\textrm{ parallel  to }M_1.
\]
Let $x \in \bd K_2$, $x\notin M_2$. Let $\pi_2:=\aff \{x, M_2\}$,  
$M:=\pi_2 \cap H$ and $\pi_1:=\aff\{M,M_1\}$. Thus 
$S(\pi_2 \cap K_2)=S_M(\pi_2 \cap K_2)=\pi_1 \cap K_1$. Hence $S(x)\in K_1$, i.e., (\ref{magia}) holds.
\end{proof}
We take a system of coordinates such that $\Sigma$ is the origin and $H$ is given by the equation $z=0$. Let $\Pi$ be a plane parallel to $H$ and let $\Pi':=S(\Pi)$, $L':=S(L)$, $a:= \Pi \cap L$, $b:= \Pi \cap L'$, $b'=S(b)$ and $\Psi:= \Pi \cap K_1$. 
 
\begin{lemma}\label{revolution}
If the points $a,b $ are interior points of the section $\Psi$, then  
$\Psi$ is a circle with center at axis $z$. 
\end{lemma}
\begin{proof}
Let $\Pi$ be plane parallel to $H$ such that  points $a$ and $b$ are interior points of the section $\Psi $. In order to prove Lemma \ref{revolution}, we are going to show that the section $\Psi$ and the points $a$ and $b$, satisfies the conditions of Lemma \ref{santo}.

First, we observe that, for each pair of parallel lines $L_1,L_2\subset  \Pi$ passing through $a$ and $b'$, respectively, it follows that 
\begin{eqnarray}\label{richard}
\len(L_1\cap K_1)=\len(L_2\cap K_2).
\end{eqnarray} 
Let $M$ be the line parallel to $L_1$ passing through $\Sigma$. We have $\pi_1 (M)= \pi_2 (M)$ and $S_M$ is the reflection in $\pi_1 (M)=\pi_2(M)$ with respect to the line $M$. By virtue of (\ref{trinidad}), $S_M(L_1)=L_2$. From here, (\ref{richard}) follows. On the other hand, we observe that, by (\ref{magia}) of Lemma \ref{isabella}, $S(L_2\cap K_2)$ is a chord of $\Pi \cap K_1$ passing through $b$, parallel to $L_1\cap K_1$. Since $\len[ S(L_2\cap K_2)]=\len[L_2\cap K_2]$, by (\ref{richard}) we have 
\[
\len(L_1\cap K_1)=\len[ S(L_2\cap K_2)] 
\]
(notice that the segment $L_1 \cap K_1$, $L_2 \cap K_2$ and $S(L_2\cap K_2)$ defined an orthogonal prism).
\end{proof}

Let $K\subset \Rn$ be a convex body and let $L$ be a line through the origin. The union of all tangent lines of $K$ parallel to $L$ will be called 
the \textit{support cylinder} of $K$ corresponding to the line $L$ and which will be denoted by $T\partial (K,L)$. The \textit{shadow boundary} of $K$ in the direction of $L$ is defined as $S\partial (K,L)=T\partial (K,L)\cap K.$ If $v\neq 0,$ then $S\partial(K,v)=S\partial (K,L),$ where $L$ is the line through the interval $[0,v]$.

\begin{remark}\label{vivaldi} 
The bodies $K_1$ and $S(K_2)$ satisfies conditions of Theorem~1 with respect to the points $p_1$, $S(p_2)$ and the plane $H$.
\end{remark}
By Lemma \ref{isabella}, $K_1=S(K_2)$. Thus, by Remark \ref{vivaldi}, $K_1$ satisfies the following property:
\begin{itemize} 
\item[(*)] For each line $M\subset  H$,
there exists a reflection $S_{M}:\Rt \rightarrow \Rt$ with respect to some plane, containing $M$, such that
\[
S_{M}(\pi_{2}(M)\cap  K_1)= \pi_{1}(M)\cap K_{1},
\]
where $\pi_{1}(M):=\aff\{p_1,M\}$ and $\pi_{2}(M):=\aff\{S(p_2),M\}$
\end{itemize}
(Notice that, since we are assuming that $L$ is not orthogonal to $H$, $p_1\not=S(p_2)$).

We recall that we denote by $M_1$ the line $\aff \{p_1,S (p_2)\}$ and we observe that, by Lemma \ref{dey}, $M_1$ is parallel to $H$. We will change the notation in such a way that now, for all line $M\subset H$, $\pi_2(M)$ will denote the plane 
$\aff\{S(p_2),M\}$.

\begin{lemma}\label{margarita}
For each plane $\Gamma$ con\-tai\-ning $ M_1$, the relation
\[
\Gamma \cap K_1 \subset S\partial (K_1,\Gamma^\perp)
\]
holds, where $\Gamma^\perp$ is the line orthogonal to $\Gamma$. Furthermore $K_1$ is a body of revolution with axis the line $M_1$.
\end{lemma}

\begin{proof}
 Let $\Gamma$ be a plane, with unit normal vector $u$, containing the line $M_1$. We denote by $m$ the intersection $\Gamma \cap H$. Let $\{m_n\}\subset H$ be a sequence of lines such that $m_n \rightarrow m$, when $n \rightarrow \infty$, and $m_n$ is not parallel to $m$. For $n=1,2,..$, the sections $\pi_1(m_n)\cap K_1$ and $\pi_2(m_n)\cap K_1$ are the sections of a cylinder $C_n$ whose generatrix are parallel to unit vector $u_n\in \mathbb{S}^{2}$ which is orthogonal to the plane of the reflection $S_{m_n}$, where $S_{m_n}:\Rt \rightarrow \Rt$ is such that $S_{m_n}(\pi_2(m_n)\cap K_1)=\pi_1(m_n)\cap K_1$. Notice that, by the strictly convexity of $K_1$, 
\begin{eqnarray}\label{luismi} 
S\partial (K_1,u_n)\subset \bd K_1\backslash C_n, 
\end{eqnarray} 
for all $n$. Since  $\pi_1(m_n) \rightarrow \Gamma$ and $\pi_2(m_n) \rightarrow \Gamma$ when $n \rightarrow \infty$, then the plane of the reflection of $S_{m_n}$ tends to $\Gamma$ when $n \rightarrow \infty$. Thus $u_n \rightarrow u$ when $n \rightarrow \infty$. 

On the one hand, by (\ref{luismi}), $S\partial (K_1,u_n)\rightarrow \Gamma \cap K_1$ when $n \rightarrow \infty$ and, on the other hand, $S\partial (K_1,u_n) \rightarrow S\partial (K_1,u)$ when $n \rightarrow \infty$. Consequently $S\partial (K_1,u)=\Gamma \cap K_1$.

In order to prove that $K_1$ is a body of revolution with axis the line $M_1$ we are going to prove that all the sections of $K_1$, perpendicular to $M_1$, are circles with center at $M_1$. Let $\Pi$ be a plane perpendicular to $M_1$ and such that $\Pi\cap \operatorname{int}(K_1)\not=\emptyset$. In order to prove that $\Pi\cap K_1$ is a circle we are going to show that, for every point $y\in \bd (\Pi\cap K_1)$, there is a supporting line $L$ of $\Pi\cap K_1$ at $y$ perpendicular to the line $L(x,y)$, where $x:=\Pi \cap M_1$ (see Fig. \ref{sb}). It is well know that such property characterizes the circle. 
\begin{figure}
    \centering
    \includegraphics[width=.88\textwidth]{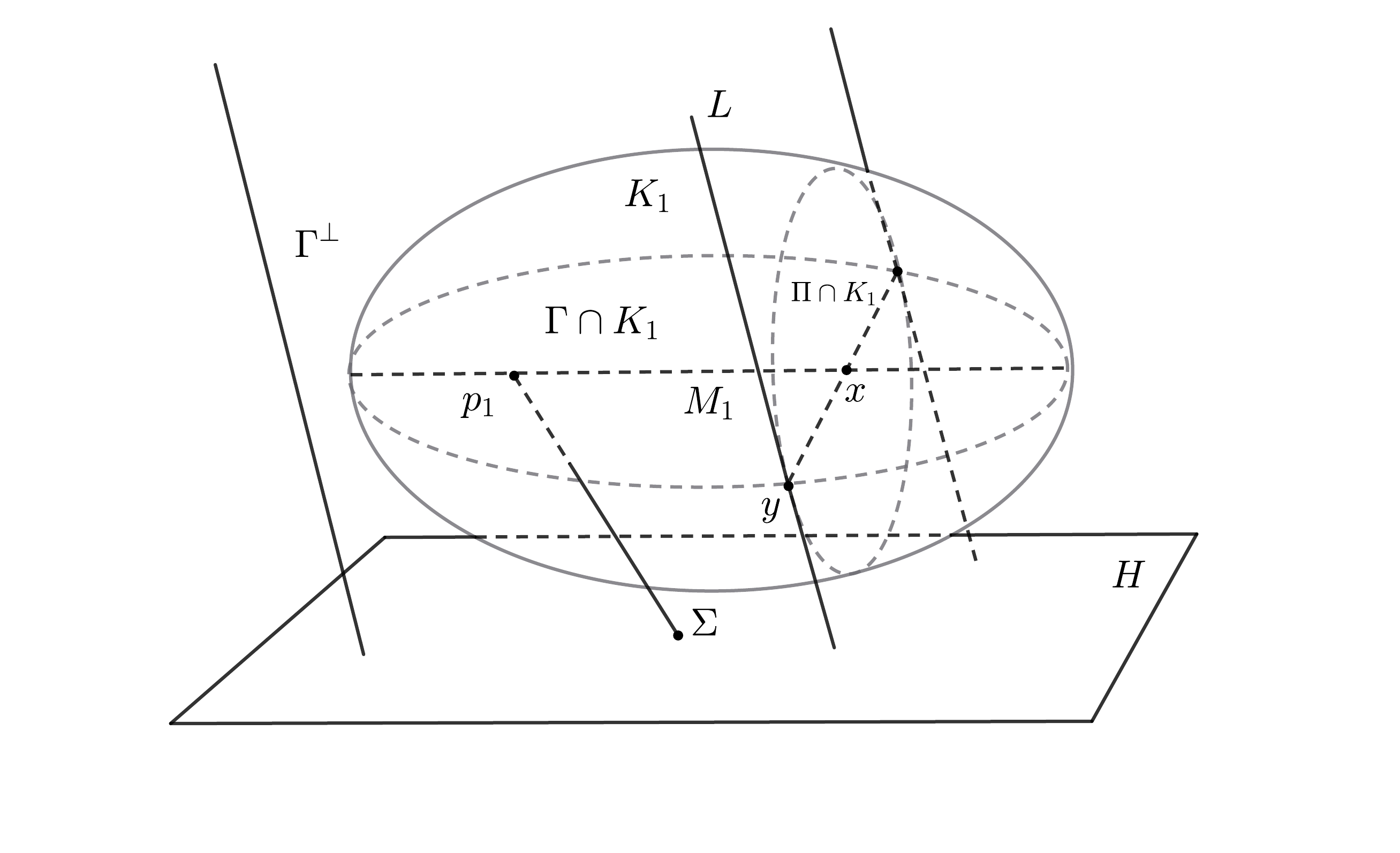}
    \caption{The body $K_1$ is a body of revolution with axis $M_1$.}
    \label{sb}
\end{figure}
Let $y\in \bd (\Pi\cap K_1)$. By the first part of the proof of Lemma \ref{margarita}, $S\partial (K_1,\Gamma^\perp)= \Gamma \cap K_1$, where $ \Gamma:=\aff\{y,M_1 \}$ and $\Gamma^\perp$ is a line  perpendicular to $\Gamma$. Thus there exists a supporting line $L$ of $K_1$ at $y$ parallel to $\Gamma^\perp$. Since $L(x,y)\subset  \Gamma$, $L(x,y)$ and $L$ are perpendicular.
\end{proof}
\begin{lemma}\label{milamores}
The body $K_1$ is a sphere.
\end{lemma}
\begin{proof}
Since $p_i \in \operatorname{int}(K_i)$, $i=1,2$, and, by Lemma \ref{isabella}, $S_H(K_2)=K_1$ it follows that $S_H(p_2)\in  \operatorname{int}(K_1)$. By Lemma \ref{revolution}, the section $H_1\cap K_1$ is a circle (we recall that $H_1$ is the plane parallel to $H$ passing through $p_1$ and containing $M_1$). On the other hand, by Lemma \ref{margarita}, $K_1$ is a body of revolution with axis the line 
$M_1$. Then $K_1$ is a body of revolution generated by the circle $H_1\cap K_1$. Hence $K_1$ is a sphere.
\end{proof}
Let $B\subset \Rt$ be a sphere with radius $r$ and center $C$
and let $q\in \operatorname{int}(B)$, $q\not=O$. We take  a system of coordinates with the origen $O$ at $C$ and with axis $x$ the line $L(O,q)$. We choose $x_0\in \R$, $0<x_0< r$, such that $q=(x_0,0,0)$. For every $k\in \R$, $0<k$, let $H_k:=\{z=-k\}$, $M_k:= H_k\cap \{x=x_0\}$, $\pi_1(M_k):=\{x=x_0\}$, $\pi_2(M_k):=\aff\{-q,M_k\}$.
\begin{lemma}\label{milpachangas}
Then, for every $k\in \R$, $0<k$, it does not exist a reflection $S:\Rt \rightarrow \Rt$ such that 
$S(\pi_2(M_k)\cap B)=\pi_1(M_k)\cap B$.
\end{lemma}
\begin{figure}
    \centering
    \includegraphics[width=.78\textwidth]{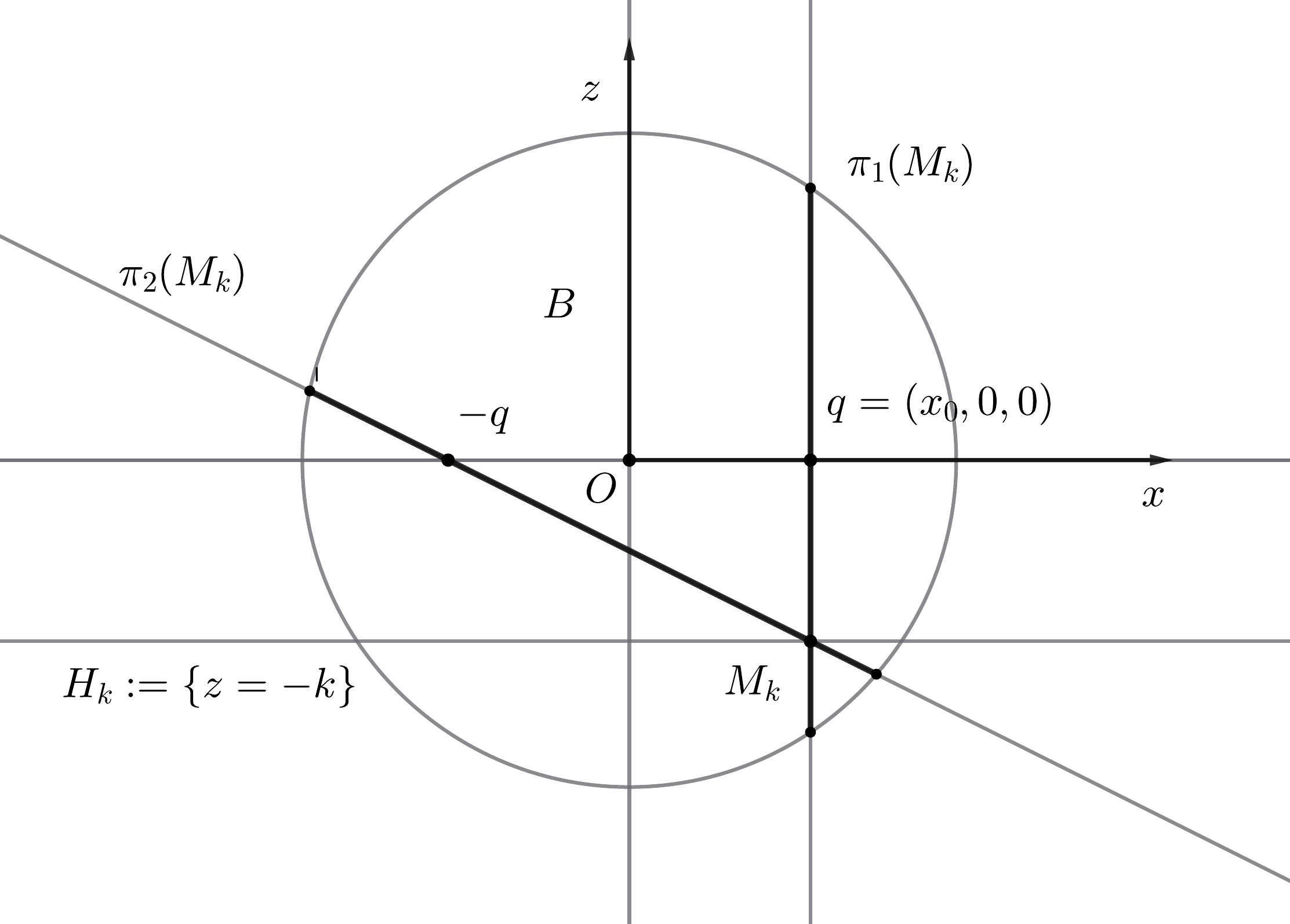}
    \caption{For every $k\in \R$, $0<k$, it does not exist a reflection $S:\Rt \rightarrow \Rt$ such that 
$S(\pi_2(M_k)\cap B)=\pi_1(M_k)\cap B$.}
    \label{sphere}
\end{figure}
\begin{proof} 
For all $k\in \R$, $0<k$, we denote by $A_k$ the area of the section $\pi_2(M_k)\cap B$ and by $A$ the area of $\{x=x_0\}\cap B$.  We observe that 
\begin{eqnarray}\label{tanga}
A <A_k \leq A_0= 2\pi r^2,
\end{eqnarray}
for all $k\in \R$, $0<k$. Thus, for every $k\in \R$, $0<k$, it does not exist a reflection $S:\Rt \rightarrow \Rt$ such that $S(\pi_2(M_k)\cap B)=\pi_1(M_k)\cap B$. Otherwise, if for some $k_0\in \R$ such reflection exist, since $S$ is an isometry, $A_{k_0}=A$ however this would contradict (\ref{tanga}). The proof of Lemma \ref{milpachangas} is complete. 
\end{proof}

\section{Lemmas for the Case $n = 3$, $p_1 \in \operatorname{int}(K_1)$, and $L \perp H$}

In the lemmas of this section we will assume that $S(p_2)\not=p_1$ and $L$ is perpendicular to $H$. 
\begin{lemma}\label{isabella2} 
$K_1$ is the image of $K_2$ under the reflection with respect to plane $H$, i.e., 
\begin{eqnarray}\label{chacha}
S(K_2)=K_1.
\end{eqnarray}
\end{lemma}

\begin{proof}
The argument of the proof is analogous to which was presented in the proof of Lemma \ref{rica}.
\end{proof}

We observe that if we replace the body $K_2$ for $S(K_2)$, then the bodies $K_1$ and $S(K_2)$ satisfies conditions of Theorem~1. By Lemma \ref{isabella2}, it follows that the body $K_1=S(K_2)$ satisfies the following property:
\begin{itemize} 
\item[(**)] For each line $M\subset  H$,
there exists a reflection $S_{M}:\Rt \rightarrow \Rt$ with respect to some plane, containing $M$, such that
\[
S_{M}(\pi_{2}(M)\cap  K_1)= \pi_{1}(M)\cap K_{1},
\]
where $\pi_{1}(M):=\aff\{p_1,M\}$ and $\pi_{2}(M):=\aff\{S(p_2),M\}$. 
\end{itemize}

\begin{lemma}\label{margarita2}
For each plane $\Gamma$, $L\subset \Gamma$, the section $\Gamma \cap K_1$ is equal to the shadow boun\-da\-ry of $S\partial(K_1,u)$, where 
$u$ is a unit vector orthogonal to $\Gamma$. Furthermore $K_1$ is a body of revolution with axis the line $L$.
\end{lemma}
\begin{proof}
The argument of the proof is analogous to which was presented in the proof of Lemma \ref{margarita}.
\end{proof} 

\section{Proof of Theorem~1 for $n = 3$}

We now complete the proof of Theorem~1 in dimension three.

Let $K_1, K_2 \subset \mathbb{R}^3$ be strictly convex bodies, and let $H \subset \mathbb{R}^3$ be a plane. Let $p_1, p_2 \in \mathbb{R}^3 \setminus H$ be distinct points, and suppose that for every line $M \subset H$, there exists a reflection $S_M$ with respect to a plane through $M$ such that
\[
S_M(\pi_2(M) \cap K_2) = \pi_1(M) \cap K_1.
\]

We distinguish three cases:

\subsection*{Case 1: $p_1 \notin K_1$}

From Lemma~4, we know that $p_2 \notin K_2$, and $L := L(p_1, p_2)$ is perpendicular to $H$. From Lemma~5, the distances from $p_1$ and $p_2$ to $H$ are equal, and thus $p_1 = S(p_2)$, where $S$ is the reflection with respect to $H$. Then, by Lemma~6, $S(K_2) = K_1$.

\subsection*{Case 2: $p_1 \in \operatorname{int}(K_1)$ and $L \not\perp H$}

Suppose that $p_{1}\in K_{1}$. Then, by Lemma \ref{dey}, $p_{2}\in K_{2}$ and 
$\|p_1-\Sigma\|=\|p_2-\Sigma\|$. By Lemma \ref{margarita}, $K_1$ is a body of revolution with axis parallel to $H$. By Lemma \ref{milamores}, $K_1$ is a sphere. Now if in the Lemma \ref{milpachangas} we make $K_1=B$, $p_1=q$, $p_2=-q$ and $H=H_k$ for some $k\in \R$, $k>0$, we get that  
it does not exist a reflection $S:\Rt \rightarrow \Rt$ such that 
$S(\pi_2(M_k)\cap B)=\pi_1(M_k)\cap B$ however this would contradict the hypothesis of Theorem~1. 

\subsection*{Case 3: $p_1 \in \operatorname{int}(K_1)$ and $L \perp H$}

First we suppose that $\|p_1-\Sigma\|=\|p_2-\Sigma\|$. Since we are assuming that $L$ is perpendicular to $H$, it follows that  $S(p_{2})=p_1$. By Lemma \ref{rica}, $S(K_{2})=K_1$.

Finally we suppose that $\|p_2-\Sigma\|<\|p_1-\Sigma\|$, i.e., $S(p_{2})\not=p_1$. 
By Lemma \ref{isabella2}, $S(K_2)=K_1$. On the one hand, by Lemma \ref{margarita2}, $K_1$ is a body of revolution.  

\bigskip
 This completes the proof of Theorem~1 for $n = 3$.
\qed

\bigskip
We thank the anonymous referees for all the comments that help to improve the paper.


\begin{thebibliography}{99}
	
	\bibitem{Aitchison1971}
	P.~W. Aitchison, C.~M. Petty, and C.~A. Rogers, 
	\newblock A convex body with a false centre is an ellipsoid, 
	\newblock \emph{Mathematika} \textbf{18} (1971), 50--59.
	
	\bibitem{Larman1974}
	D.~G. Larman,
	\newblock A note on the false centre problem, 
	\newblock \emph{Mathematika} \textbf{21} (1974), 216--227.
	
	\bibitem{MontejanoMorales2003}
	L.~Montejano and E.~Morales, 
	\newblock Polarity in convex bodies: Characterizations of ellipsoids, 
	\newblock \emph{Mathematika} \textbf{50} (2003), 63--72.
	
	\bibitem{MontejanoMorales2007a}
	L.~Montejano and E.~Morales-Amaya, 
	\newblock Variations of classic characterizations of ellipsoids and a short proof of the false centre theorem, 
	\newblock \emph{Mathematika} \textbf{54} (2007), 35--40.
	
	\bibitem{MontejanoMorales2007b}
	L.~Montejano and E.~Morales-Amaya, 
	\newblock A shaken false centre theorem, 
	\newblock \emph{Mathematika} \textbf{54} (2007), 41--46.
	
	\bibitem{Rogers1965}
	C.~A. Rogers, 
	\newblock Sections and projections of convex bodies, 
	\newblock \emph{Portugaliae Math.} \textbf{24} (1965), 99--103.
	
	\bibitem{Rogers1981}
	C.~A. Rogers, 
	\newblock An equichordal problem, 
	\newblock \emph{Geom. Dedicata} \textbf{10} (1981), 73--78.
	
\end{thebibliography}
\end{document}